\documentclass[11pt]{article}
\usepackage[latin1]{inputenc}
\usepackage{amsfonts,amsmath,amsthm,amssymb,mathrsfs}
\usepackage{fullpage,fancyhdr,hyperref}
\usepackage{enumitem,graphicx,xcolor,booktabs,blkarray}
\usepackage[font={footnotesize,sf}]{caption}

\def\ds{\displaystyle}

\def\G{\mathcal{G}}
\def\A{\mathcal{A}}

\def\bx{\mathbf{x}}

\newtheorem{theorem}{Theorem}[section]

\newtheorem{remark}[theorem]{Remark}
\newtheorem{proposition}[theorem]{Proposition}
\newtheorem{lemma}[theorem]{Lemma}
\newtheorem{definition}[theorem]{Definition}

\title{Chordal graphs with bounded tree-width}
\author{Jordi Castellv\'i \and Michael Drmota \and Marc Noy \and Cl\'ement Requil\'e}

\AtEndDocument{%
\par
\bigskip
    \noindent\textsc{Jordi Castellv\'i}\\
    Departament de Matem\`atiques de la Universitat Polit\`ecnica de Catalunya (UPC), Barcelona, Spain.\\
    \textit{E-mail}: \texttt{jordi.castellvi@upc.edu}
\par
\bigskip
    \noindent\textsc{Michael Drmota}\\
    Institute for Discrete Mathematics and Geometry of the Technische Universit\"at Wien, Austria.\\
    \textit{E-mail}: \texttt{michael.drmota@tuwien.ac.at}
\par
\bigskip
    \noindent\textsc{Marc Noy}\\
    Departament de Matem\`atiques and Institut de Matem\`atiques (IMTech) de la Universitat Polit\`ecnica de Catalunya (UPC), and Centre de Recerca Matem\`atica (CRM), Barcelona, Spain.\\
    \textit{E-mail}: \texttt{marc.noy@upc.edu}
\par
\bigskip
    \noindent\textsc{Cl\'ement Requil\'e}\\
    Departament de Matem\`atiques and Institut de Matem\`atiques (IMTech) de la Universitat Polit\`ecnica de Catalunya (UPC), Barcelona, Spain.\\
    \textit{E-mail}: \texttt{clement.requile@upc.edu}
}

\begin{document}
\maketitle

\begin{abstract}
Given $t\ge 2$ and $0\le k\le t$, we prove that the number of labelled $k$-connected chordal graphs with $n$ vertices and tree-width at most $t$ is asymptotically $c n^{-5/2} \gamma^n n!$, as $n\to\infty$, for some constants $c,\gamma >0$ depending on $t$ and $k$.
Additionally, we show that the number of $i$-cliques ($2\le i\le t$) in a uniform random $k$-connected chordal graph with tree-width at most $t$ is normally distributed as $n\to\infty$.

The asymptotic enumeration of  graphs of tree-width at most $t$ is wide open for $t\ge 3$.
To the best of our knowledge, this is the first non-trivial class of graphs with bounded tree-width where the asymptotic  counting problem is solved. 
Our starting point is the work of Wormald [Counting Labelled Chordal Graphs, \textit{Graphs and Combinatorics} (1985)], were an algorithm is developed to obtain the exact number of labelled chordal graphs on $n$ vertices.
\end{abstract}


\section{Introduction}\label{sec:intro}

Tree-width is a fundamental parameter in structural and algorithmic graph theory, as illustrated for instance in \cite{CFKLMPPS15}. 
It can be defined in terms of tree-decompositions or equivalently in terms of  $k$-trees.
A $k$-tree is defined recursively as either a complete graph on $k+1$ vertices or a graph obtained by adjoining a new vertex adjacent to a $k$-clique of a $k$-tree. 
The tree-width of a graph $\Gamma$ is  the minimum $k$ such that $\Gamma$ is a subgraph of a $k$-tree.
In particular, $k$-trees are the maximal graphs with tree-width at most $k$.
The number of $k$-trees on $n$ labelled vertices was independently shown \cite{BP69,M69} to be
\begin{equation}\label{eq:estimate_ktrees}
    \binom{n}{k}(k(n-k) + 1)^{n-k-2} = \frac{1}{\sqrt{2\pi}\, k!\, k^{k+2}}\, n^{-5/2}\, (ek)^n\, n!\, (1 + o(1)),
\end{equation}
where the estimate holds for $k$ fixed and $n\to\infty$.
However, there are relatively few results on the enumeration of graphs of given tree-width or on properties of random graphs with given tree-width.
Graphs of tree-width one are forests (acyclic graphs) and their enumeration is a classical result, while graphs of tree-width at most two are series-parallel graphs and were first counted in \cite{BGKN07}.
The problem of counting graphs of tree-width three is still open.
From now on, we will use $t$ to denote the tree-width while $k$ will denote the connectivity of a graph. 
All graphs we consider are labelled. 

Given that tree-width is non-increasing under taking minors, the class of graphs with tree-width at most $t$ is `small' when $t$ is fixed, in the sense that the number $g_{n,t}$ of labelled graphs with $n$ vertices and tree-wdith at most $t$ grows at most like $c^n n!$ for some $c>0$ depending on $t$ (see \cite{NSTW06,DN10}). 
The best bounds known for $g_{n,t}$ are, up to lower order terms, 
\begin{equation*}
    \left(\frac{2^tt n}{\log t}\right)^n 
    \le g_{n,t} \le ( 2^tt n)^n.
\end{equation*}
The upper bound follows by considering all possible subgraphs of $t$-trees, and the lower bound uses a suitable construction developed in \cite{BNS18}.    
In the present work we determine the asymptotic number of labelled chordal graphs with tree-width at most~$t$, following the approach in \cite{FS09} and \cite{D09}, and based on the analysis of systems of equations satisfied by generating functions.

A graph is chordal if every cycle of length greater than three contains at least one chord, that is, an edge connecting non-consecutive vertices of the cycle.
Chordal graphs have been extensively studied in structural graph theory and graph algorithms (see for instance \cite{G04}), but not so much from the point of view of enumeration. 
Wormald \cite{W85} used generating functions to develop a method for finding the exact number of chordal graphs with $n$ vertices for a given value of $n$. 
It is based on decomposing chordal graphs into $k$-connected components for each $k \ge 1$. 
As remarked in \cite{W85}, it is difficult to define the $k$-connected components of arbitrary graphs for $k>3$, but for chordal graphs they are well defined. 
Wormald  introduces generating functions $C_k(x)$ for $k$-connected chordal graphs and finds an equation linking $C_k(x)$ and $C_{k+1}(x)$ reflecting the decomposition into $k$-connected components. 
This is precisely the starting point of our work.

An important parameter in \cite{W85} is the size $w$ of the largest clique. 
For chordal graphs one can show that the tree-width is equal to $w-1$, hence bounding the tree-width by $t$ is the same as bounding $w$ by $t+1$. 
The parameter $w$ also plays a substantial r\^ole in showing that almost all chordal graphs are split graphs \cite{BRW85}, which in turn implies that the number of chordal graphs with $n$ vertices is of order $2^{n^2/4}(1 + o(1))$. 
  
For fixed $n,t\ge 1$ and $0\le k\le t$, let $\mathcal{G}_{t,k,n}$ be the set of $k$-connected chordal graphs with $n$ labelled vertices and tree-width at most $t$.
Our two main results are the following.

\begin{theorem}\label{thm:enumeration}
For $t\ge 1$ and $0\le k\le t$, there exist  constants $c_{t,k} > 0$ and $\gamma_{t,k}>1$ such that
\begin{equation*}
    |\mathcal{G}_{t,k,n}| = c_{t,k}\, n^{-5/2}\, \gamma_{t,k}^n\, n! \, \left(1 + O\left(n^{-1}\right)\right)
    \qquad\text{as } n\to\infty.
\end{equation*}
\end{theorem}

Note that by setting $t=k$ in Theorem \ref{thm:enumeration} one recovers the general form of the asymptotic estimate of the number of $k$-trees \eqref{eq:estimate_ktrees}.
The fact that $\gamma_{k,k}=ek$ is reproven in Section \ref{sec:analysis}.
In principle, for fixed $t$ and $k$ the constants $c_{t,k}$ and $\gamma_{t,k}$ can be computed, at least approximately.
Table \ref{tab:values_rho} in Section \ref{sec:conclusion}  displays approximations of $1/\gamma_{t,k}$ for $1\le k\le t\le 7$.

\begin{theorem}\label{thm:limit_law}
Let $t\ge 1$, $0\le k\le t$. For $i\in \{2,\ldots,t\}$ let $X_{n,i}$ denote the number of $i$-cliques in a uniform random graph in $\mathcal{G}_{t,k,n}$, and set ${\bf X}_n = (X_{n,2},\ldots, X_{n,t})$.
Then ${\bf X}_n$ satisfies a multivariate central limit theorem, that is, as $n\to\infty$ we have 
\begin{equation*}
    \frac{1}{\sqrt n} \left( {\bf X}_n - \mathbb{E}\, {\bf X}_n\right) \overset{d}\to N(0,\Sigma),
    \qquad\text{with}\quad
    \mathbb{E}\, {\bf X}_n \sim \alpha n \quad\mbox{and} \quad \mathbb{C}\mathrm{ov}\, {\bf X}_n \sim \Sigma n,
\end{equation*}
and where $\alpha$ is a $(t-1)$-dimensional vector of positive numbers and $\Sigma$ is a $(t-1)\times(t-1)$-dimensional positive semi-definite matrix. 
\end{theorem}

Let us mention that more structural asymptotic results can be expected.
Notably, the class 
of chordal graphs with tree-width at most $t$ is \textit{subcritical} in the sense of \cite{DFKKR11}, as further discussed in the concluding Section~\ref{sec:conclusion}. 
It follows from \cite{PSW16} that the uniform random connected chordal graph with tree-width at most $t$ 
 with distances rescaled by $1/\sqrt{n}$ admits the \textit{Continuum Random Tree} (CRT) \cite{A91} as a scaling limit, multiplied by a constant that depends on $t$.

The proofs of Theorems~\ref{thm:enumeration} and \ref{thm:limit_law} are based on a recursive decomposition of chordal graphs translated in the combinatorial language of generating functions, the so-called \textit{symbolic method} \cite{FS09}, into a system of functional equations explicited in Section~ \ref{sec:combi} and analysed in Section~\ref{sec:analysis} by means of analytic methods~\cite{FS09}.
More precisely, in Subsection~\ref{subsec:unicity_components} we show the unicity of the decomposition of chordal graphs into their $k$-connected components.
This is translated in Subsection~\ref{subsec:system_equations} into a system of funtional equations satisfied by exponential generating functions encoding the numbers of graphs in each class.
In Subsection~\ref{subsec:combinatorial_integration} we prove an alternative encoding of an integral operator, which is instrumental for computing the numerical values in Table~\ref{tab:values_rho}.
The subsequent asymptotic analysis is rather delicate as it involves singularity functions depending on several variables. 
Our notions of \textit{proper} and \textit{fully movable singularity functions} are key ingredients, and are defined along several technical notions in Subsection~\ref{subsec:definitions_singularity}.
Some useful propeties are proven in Subsection~\ref{subsec:tansfer_propeties}, before embarking in the proofs of the main results in Subsection~\ref{subsec:proofs}.


\section{Decomposition of chordal graphs}\label{sec:combi}

All graphs considered in this work will be simple and labelled, that is, with vertex-set $[n]$.
Let $\Gamma$ be a graph and $k\geq 1$.
A~$k$-separator of $\Gamma$ is a subset of $k$ vertices whose removal disconnects $\Gamma$.
The graph $\Gamma$ is said to be $k$-connected if it contains no $i$-separator for $i\in[k-1]$.
Notice that with this definition we consider the complete graph on $k$ vertices to be $k$-connected, for any $k\ge 1$, contrary to the usual definition of connectivity (see for instance \cite{D16}).
A $k$-connected component of $\Gamma$ is a $k$-connected subgraph that is maximal, in term of subgraph containment, with that property.

An essential consequence of chordality is that $k$-connected chordal graphs admit a unique decomposition into $(k+1)$-connected components through its $k$-separators.
This is the subject of the next section.

\subsection{Unicity of the components}\label{subsec:unicity_components}

The following well-known result will play a central role in our definition of the decomposition of a chordal graph into $k$-connected components.
For completeness we provide a short proof.
\begin{proposition}[Dirac \cite{D61}]\label{prop:dirac}
In a chordal graph every minimal separator is a clique. 
\end{proposition}
\begin{proof}
Let $\Gamma$ be a chordal graph and let $S$ be a minimal separator with at least two vertices. 
Suppose for contradiction that there are $u,v \in S$ such that $uv$ is not an edge. 
Let $A$ and $B$ be different components of $\Gamma - S$, and consider shortest paths $P$ and $Q$ between $x$ and $y$ whose inner vertices are, respectively, in $A$ and $B$. 
Then the concatenation of $P$ and $Q$ is a chordless cycle of length at least four, contradicting the fact that $\Gamma$ is chordal. 
\end{proof}

For the rest of this section, we fix $k\geq 1$.
\begin{definition}\label{def:slices}
Let $\Gamma$ be a $k$-connected chordal graph with a $k$-separator $S$, and, for $m\ge 1$, let $C_1, \ldots, C_m$ be the (possibly empty) connected components of $\Gamma - S$.
Then, for $i\in [m]$, the induced subgraphs $\Gamma_i = \Gamma[V(C_i)\cup S]$ will be called the slices of $\Gamma$, obtained after cutting $\Gamma$ through $S$.
\end{definition}
Note that by Proposition \ref{prop:dirac}, $S$ is a clique of $\Gamma$.
Furthermore, as each slice $\Gamma_i$ ($i\in[m]$) contains a copy of $S$, $\Gamma$ can be obtained by identifying together these $m$ copies of $S$.
This operation will be called \textit{gluing} through $S$.
The next Proposition \ref{prop:separating_sets_in_slices} implies that the slices $\Gamma_i$ are $k$-connected. 
\begin{proposition}\label{prop:separating_sets_in_slices}
    Let $\Gamma$ be a $k$-connected chordal graph with a $k$-separator $S$ inducing the slices $\Gamma_1, \ldots, \Gamma_m$.
    If, for some $i\in [m]$, $\Gamma_i$ has a separator $T$, then $T$ is also a separator of $\Gamma$.
    Furhtermore, $T\neq S$ is a $k$-separator of $\Gamma$ if it is a separator of the slice $\Gamma_i$ it belongs to.
\end{proposition}
\begin{proof}
If $S\subseteq T$, the first claim is direct. 
Suppose now that $v\in S \setminus T$. 
Then, $T$ separates $v$ from some other vertex $w\in \Gamma_i$, because otherwise every vertex would be reachable from $v$ and $T$ would not be a separator. 
Since the vertices in $S$ form a clique, $w$ is in fact separated from all vertices in $S\setminus T$.
Then, in $\Gamma$, the same set $T$ also separates $w$ from $S\setminus T$.

For the second claim, observe that since all the vertices and edges in $\Gamma$ are also in some $\Gamma_i$ and vice-versa, a set of vertices $T\neq S$ is a subclique of $\Gamma$ if and only if it is a subclique of some slice $\Gamma_i$.
Now suppose to a contradiction that $\Gamma_i - T$ is connected. 
Since, for $j\neq i$, $T$ is not entirely contained in any other slice $\Gamma_j$, and $\Gamma_j$ is $k$-connected, it follows that $\Gamma_j - T$ is also connected. 
In particular, every vertex in $\Gamma - T$ is reachable from some vertex $v\in S\setminus T$, a contradiction.
\end{proof}

Consider now the set of slices obtained after cutting $\Gamma$ through all its $k$-separators.
Because they contain no $k$-separators, all these slices are $(k+1)$-connected and form in fact the $(k+1)$-connected components of $\Gamma$.
They are well defined thanks to the following result. 
\begin{proposition}\label{prop:cuts_commute}
Let $\Gamma$ be a $k$-connected chordal graph with $k$-separators $S$ and $T$.
Then the slices $\Gamma_1,\ldots,\Gamma_m$ obtained after cutting $\Gamma$ first through $S$ then through $T$ can be characterised as follows:
\begin{enumerate}
    \item[(i)]
    Let $i\in[m]$. 
    For each connected component $C_i$ of $\Gamma - (S\cup T)$ there will be a slice $\Gamma_i$ with vertex set $V_i$, in such a way that $\Gamma_i = \Gamma[V_i]$. 
    The set $V_i$ contains the vertices in $C_i$, and also contains the vertices of $S$ (resp. $T$) if there is some vertex in $S\setminus T$ (resp. $T\setminus S$) that has a neighbour in $C_i$.

    \item[(ii)]
    If none of the slices described in (i) contains the vertices in both $S$ and $T$, there will be an additional slice $\Gamma_{m+1} = \Gamma[S\cup T]$, and these are all the possible slices.
\end{enumerate}
In particular, doing the cuts first through $T$ and then through $S$ results in the same slices.
\end{proposition}
\begin{proof}
We first consider the slices obtained after cutting only through $S$. 
For $1\leq i \leq m_S$, each of the slices $\Gamma^S_i$ corresponds to a connected component $C^S_i$ of $\Gamma - S$. 
Among these slices there is only one containing $T$ as a subclique, say $\Gamma^S_1$, because the vertices in $T\setminus S$ necessarily belong to the same component $C^S_1$. 
Therefore $\Gamma^S_1$ is the unique slice that will be cut through $T$ while the others stay unchanged. 
For $2\leq i\leq m_S$, each of the other slices $\Gamma^S_i$ contains the vertices of $C^S_i$, which is certainly a connected component of $\Gamma - (S\cup T)$, and the vertices in $S$, but contains no vertex in $T\setminus S$. 
This agrees with $(i)$ because all the vertices in $S$ have a neighbour in $C^S_i$, since $S$ is a minimal separator, while the vertices in $T\setminus S$ have no neighbours in $C^S_i$.
    
We now consider the slices obtained after cutting $\Gamma^S_1$ through $T$. 
Again, for $1\leq i\leq m_T$ each of these slices corresponds to a connected component $C^T_i$ of $\Gamma^S_1 - T$, denoted by $\Gamma^T_i$. 
Among these slices, there is only one containing $S$ as a subclique, say $\Gamma^T_1$, because the vertices in $T\setminus S$ necessarily belong to the same component $C^T_1$. 
The rest of the slices contain no vertex in $S\setminus T$. 
In fact, they are analogous to the slices $\Gamma^S_i$, $2\leq i\leq m_S$, and they agree with $(i)$ for the same reasons. 

There are two possibilities for $C_1^T$: either it contains vertices other than the ones in $S\setminus T$ or it does not. 
In the first case, observe that $C^T_1 - S$ is connected. 
Indeed, if this was not the case $\Gamma^S_1$ would have $S\cup T$ as a separator, while neither $S$ nor $T$ are separators. 
But then there would be a minimal $k$-separator of $\Gamma^S_1$ that is not a clique, which is not possible. 
Therefore, $C^T_1 - S$ is a connected component of $\Gamma - (S\cup T)$, which has some neighbour in $S\setminus T$ and also in $T\setminus S$, since $C^S_1$ is connected. 
This also agrees with $(i)$. 
On the other hand, if $C^T_1$ contains no vertices other than the ones in $S\setminus T$, then it is not a component of $\Gamma - (S\cup T)$ and we are in the case $(ii)$.

Since this characterisation does not depend on the order of the cuts, the last claim follows.
\end{proof}

The $k$-connected components of $\Gamma$ are thus the maximal $k$-connected subgraphs, since, for $i\in [k-1]$, there is a single way of cutting through all the $i$-separators.
And we can uniquely define the $2$-connected components of a connected chordal graph, the $3$-connected components of these $2$-connected components, the $4$-connected components of these $3$-connected components, and so on.
An illustration is given in Figure \ref{fig:components}.
\begin{figure}[htb]
    \centering
    \includegraphics[width=0.5\textwidth]{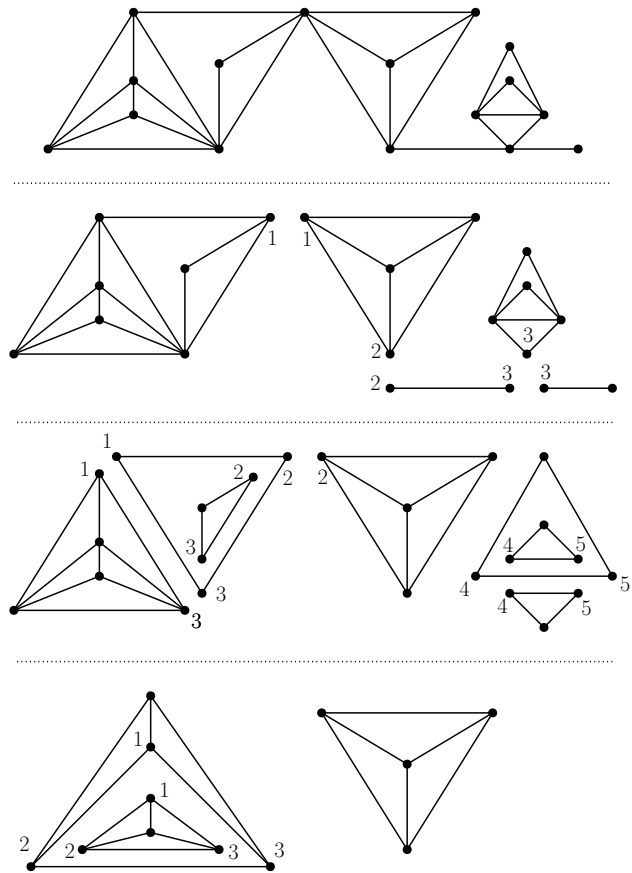}
    \caption{Decomposition of a connected graph with tree-width 3 into its 2, 3 and 4-connected components. Vertices with the same label are identified.}
    \label{fig:components}
\end{figure}
This decomposition is the generalisation of the well-known decomposition of a connected graph into blocks (i.e. maximal 2-connected components).
And it induces, as shown in the next section, a system of functional equations satisfied by the generating function counting chordal graphs of tree-width at most $k$.

\subsection{Functional equations for the generating functions}\label{subsec:system_equations}

Our main tool in this section is the symbolic method, where we will translate the combinatorial decompositions derived in the previous section into functional equations satisfied by the associated generating functions.
We refer the reader to \cite{FS09} for details on the symbolic method, but let us recall that the disjoint union of two classes of graphs ${\cal A}$ and ${\cal B}$, with respective generating function $A(x)$ and $B(x)$, is denoted by the class ${\cal A} \sqcup {\cal B}$ and admits generating function $A(x) + B(x)$; while the class composed of sets of graphs in ${\cal A}$ has generating function $\exp(A(z))$.

For the rest of this section, we fix some $t\ge 1$ and let $\mathcal{G}$ be the family of chordal graphs with tree-width at most $t$. 
For a graph $\Gamma \in \mathcal{G}$ and $j\in[t]$, let us denote by $n_j(\Gamma)$ the number of $j$-cliques of $\Gamma$.
In the rest of the paper, we will write $\bx$ as a short-hand for $x_1,\ldots, x_t$, and define the multivariate (exponential) generating function associated to $\mathcal{G}$ to be
\begin{equation*}
    G(\bx) = G(x_1, \dots, x_t) = \sum_{\Gamma \in \mathcal{G}} \frac{1}{n_1(\Gamma)!} \prod_{j=1}^{t} x_j^{n_j(\Gamma)},
\end{equation*}
Let $g_n$ denote the number of chordal graphs with $n$ vertices and tree-width at most $t$. 
Then,
\begin{equation*}
    G(x, 1, \dots, 1) = \sum_{n\geq 1} \frac{g_n}{n!} x^n.
\end{equation*}
For $0\le k\le t+1$, let $\mathcal{G}_k$ be the family of $k$-connected chordal graphs with tree-width at most $t$ and $G_k(\bx)$ be the associated generating function.
In particular, we have 
\begin{equation}\label{eq:G_t+1}
    G_{t+1}(\bx) = \frac{1}{(t+1)!}\prod_{j\in[t]}x_j^{\binom{t+1}{j}}.
\end{equation}

For other values of $k$, we need to consider graphs rooted at a clique.
Rooting the graph $\Gamma\in \mathcal{G}_k$ at an $i$-clique means distinguishing one $i$-clique $K$ of $\Gamma$ and choosing an ordering of (the labels of) the vertices of $K$.
In order to avoid over-counting, we will discount the subcliques of $K$.
Let $i\in [k]$ and define $\mathcal{G}_{k}^{(i)}$ to be the family of $k$-connected chordal graphs with tree-width at most $t$ and rooted at an $i$-clique.
Let then $G_k^{(i)}(\bx)$ be the associated generating function, where now for $1\le j \le i$ the variables $x_j$ mark the number of $j$-cliques that are not subcliques of the root.
\begin{figure}[h]
    \centering
    \includegraphics[width=0.55\textwidth]{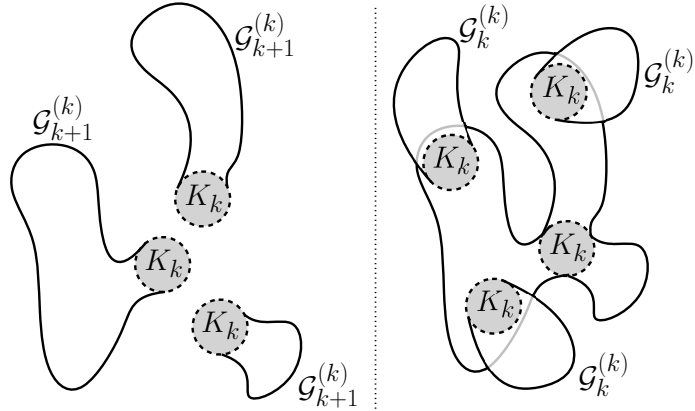}
    \caption{Recursive decomposition of a $k$-connected chordal graph into $(k+1)$-connected components. On the left: a set of graphs from $\mathcal{G}_{k+1}^{(k)}$. On the right: the roots of the graphs are identified and, then, all other $k$-cliques are recursively substituted by graphs in $\mathcal{G}_{k}^{(k)}$ to obtain a grap in $\mathcal{G}_{k}^{(k)}$ whose $(k+1)$-connected components containing the root are the graphs in the set. A label $\mathcal{G}_{k+1}^{(k)}$ or $\mathcal{G}_{k}^{(k)}$ next to a blob means that the blob is any graph from the corresponding class.}
    \label{fig:recursive}
\end{figure}
\begin{lemma}\label{lem:system}
Let $k \in[t]$.
Then the following equations hold:
\begin{align}
    G_{k+1}^{(k)}(\bx) &= k!\prod_{j=1}^{k-1}x_j^{-\binom{k}{j}}\frac{\partial}{\partial x_k} G_{k+1}(\bx) \label{eq:rooting}, \\
    G_k^{(k)}(\bx) &= \exp \left( G_{k+1}^{(k)}\big( x_1, \ldots, x_{k-1}, x_kG_k^{(k)}(\bx), x_{k+1}, \ldots, x_t \big) \right) \label{eq:recursive}, \\
    G_k(\bx) &= \frac{1}{k!}\prod_{j=1}^{k-1}x_j^{\binom{k}{j}}\int G_k^{(k)}(\bx)\,\, dx_k \label{eq:integrating}.
\end{align}
\end{lemma}
\begin{proof}
As per the symbolic method, taking the derivative of a generating function with respect to the variable $x_i$ amounts to rooting a graph at an $i$-clique, while taking the integral will correspond to ``unrooting''.
Thus, both Equations \eqref{eq:rooting} and \eqref{eq:integrating} follow from the definition.

On the other hand, Equation \eqref{eq:recursive} is derived from the decomposition of $k$-connected chordal graphs into their $(k+1)$-connected components, and is illustrated in Figure \ref{fig:recursive}.
Indeed, a $k$-connected chordal graph rooted at a $k$-clique $K$ is obtained by gluing through $K$ a set of $(k+1)$-connected graphs $\Gamma_1, \ldots, \Gamma_m$ containg $K$, and then further gluing some $k$-connected chordal graph at each $k$-clique of $\Gamma_i$, other than $K$, for all $i\in[m]$.
Recall that the $k$-clique is itself considered to be $k$-connected and that the variable $x_k$ marks the number of $k$-cliques.
The graphs are rooted at ordered cliques in order to ensure that the gluing process is made in all possible ways.

Thus, following again the symbolic method, the recursive process of gluing through $k$-cliques is encoded by the substitution of $x_k$ by $x_kG_k^{(k)}$, and Equation \eqref{eq:recursive} follows from the fact that sets of $(k+1)$-connected graphs are encoded by taking the exponential of the generating function of $(k+1)$-connected chordal graphs rooted at a $k$-clique.
\end{proof}

Finally, the fact that a graph is the set of its connected components can be translated as
\begin{equation*}
    G(\bx) = G_0(\bx) = \exp(G_1(\bx)),
\end{equation*}
Observe then that one can derive $G_0(\bx)$ from $G_{t+1}(\bx)$ by successively using Identities \eqref{eq:rooting}, \eqref{eq:recursive} and \eqref{eq:integrating} from Lemma \ref{lem:system}, as illustrated in Figure \ref{fig:schema}.
Furthermore, in the steps where Identity \eqref{eq:integrating} is used, one needs to compute a formal integral.
An alternative is to use the dissymmetry theorem for tree-decomposable classes, as presented in Proposition~\ref{prop:species}.
This is the purpose of the next section.
\begin{figure}[h]
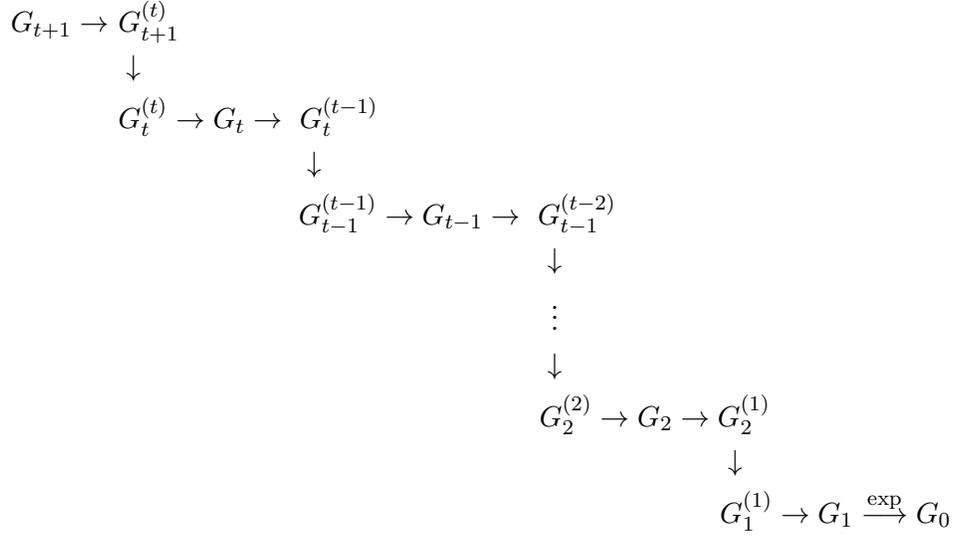

\centering
\begin{align*}
    G_{t+1} \rightarrow \,\, & G_{t+1}^{(t)} \\ 
    & \downarrow \\
    & G_{t}^{(t)} \rightarrow G_{t} \rightarrow \,\, G_{t}^{(t-1)} \\
    & \hspace{2.4cm} \downarrow \\
    & \hspace{2.4cm} G_{t-1}^{(t-1)} \rightarrow G_{t-1} \rightarrow \,\, G_{t-1}^{(t-2)} \\
    & \hspace{5.6cm}\downarrow \\
    & \hspace{5.75cm} \vdots \\
    & \hspace{5.6cm} \downarrow \\
    & \hspace{5.6cm} G_2^{(2)} \rightarrow G_2 \rightarrow G_2^{(1)} \\
    & \hspace{8cm} \downarrow \\
    & \hspace{8cm} G_1^{(1)} \rightarrow G_1 \stackrel{\exp}{\longrightarrow} G_0
\end{align*}
    \caption{The schema to derive $G_0(\bx)$ from $G_{t+1}(\bx)$.}
    \label{fig:schema}
\end{figure}

\subsection{Combinatorial integration}\label{subsec:combinatorial_integration}

In this section, we prove an alternative equation for \eqref{eq:integrating} which does not involve an integral operator.
It is useful when computing the numerical values in Table \ref{tab:values_rho}.

A class of graphs ${\cal A}$ is said to be \textit{tree-decomposable} if for each graph $\Gamma\in {\cal A}$ we can associate in a unique way a tree $\tau(\Gamma)$ whose nodes are distinguishable, for instance by using the labels.
Let us stress that this decomposition is in principle different to the one induced by the classical definition of tree-width.
In our case, to each $\Gamma\in\mathcal{G}_k$ different from the complete graph on $k$ vertices, we associate a unique tree $\tau(\Gamma)$ as follows. 
The tree $\tau(\Gamma)$ admits two different types of nodes, namely $b$ and $c$. 
Nodes of type $b$ represent the $(k+1)$-connected components of $\Gamma$, while those of type $c$ represent the $k$-cliques of $\Gamma$ through which the $(k+1)$-connected components are glued together.
An example of the decomposition of a chordal graph $\Gamma$ of bounded tree-width and its associated tree $\tau(\Gamma)$ is depicted in Figure~\ref{fig:dissymmetry}.

Let ${\cal A}_{\bullet}$ denote the class of graphs in ${\cal A}$ where a node of $\tau(\Gamma)$ is distinguished.
Similarly, ${\cal A}_{\bullet - \bullet}$ is the class of graphs in ${\cal A}$ where an edge of $\tau(\Gamma)$ is distinguished, and ${\cal A}_{\bullet \rightarrow \bullet}$ those where an edge $\tau(\Gamma)$ is distinguished and given a direction.
As presented in \cite{CFKS08}, the \textit{dissymmetry theorem for tree-decomposable classes} is a generalisation of the well-known \textit{dissymetry theorem for trees} of \cite{BLL97}, and allows one to express the class of unrooted graphs in $\A$ in terms of the rooted classes.
\begin{proposition}[Dissymmetry Theorem \cite{CFKS08}]\label{prop:species}
    Let ${\cal A}$ be a tree-decomposable class, then
    \begin{equation*}
        {\cal A} \sqcup {\cal A}_{\bullet \rightarrow \bullet} \simeq {\cal A}_{\bullet} \sqcup {\cal A}_{\bullet - \bullet},
    \end{equation*}
    where $\simeq$ is a bijection preserving the number of nodes.
    In particular, if the encoding trees have no adjacent nodes of the same type then we have
    \begin{equation*}
        {\cal A} \sqcup {\cal A}_{\bullet - \bullet} \simeq {\cal A}_\bullet.
    \end{equation*}
\end{proposition}
\begin{figure}
    \centering
    \includegraphics[width=0.9\textwidth]{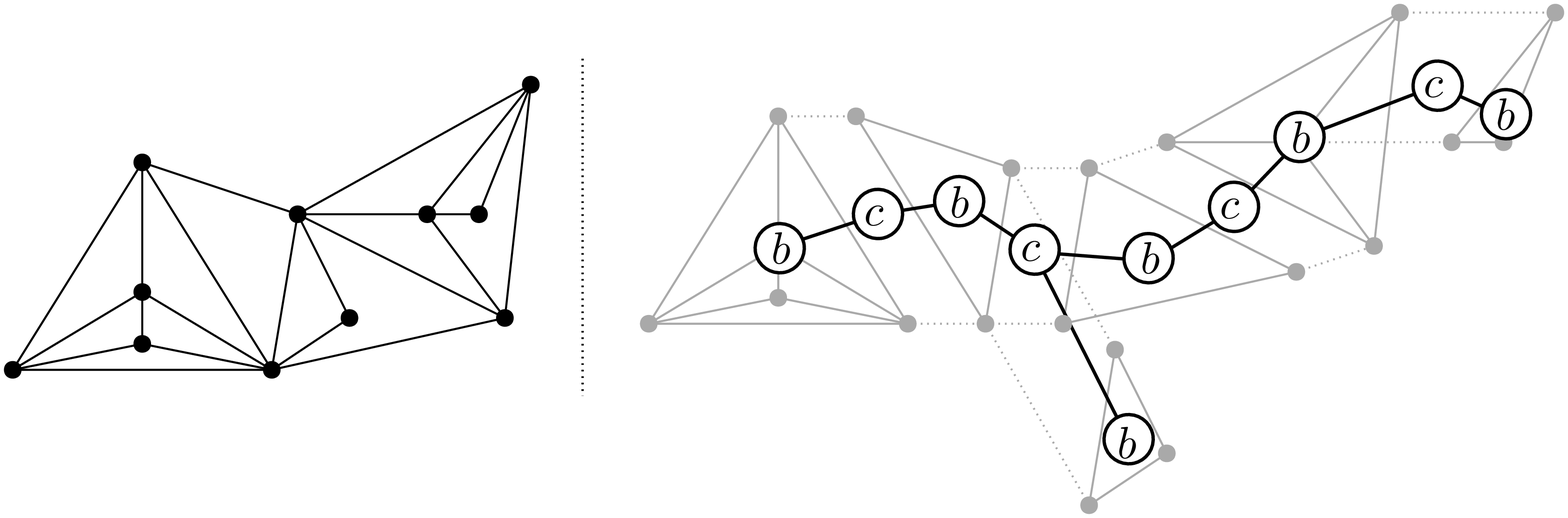}
    \caption{Tree-decomposition (right) associated to a $2$-connected chordal graph (left) of tree-width 3.
    The $b$-nodes of the tree represent the $(k+1)$-connected components of the graph, and $c$-nodes represent the $k$-cliques through which the $(k+1)$-connected components are glued together.}
    \label{fig:dissymmetry}
\end{figure}
Next, we make use of this decomposition to obtain, via the above Proposition, the generating function of unrooted chordal graphs of bounded tree-width.
\begin{lemma}\label{lem:combinatorial_integration}
Let $k\in[t]$.
Then the following equation holds:
\begin{equation}\label{eq:dissymetry}
    \begin{split}
        G_k(\bx) 
        & = G_{k+1}\big(x_1, \ldots, x_{k-1}, x_k G_k^{(k)}(\bx), x_{k+1}, \ldots, x_t\big) \\
        & \qquad + \frac{1}{k!} \prod_{j\in[k]} x_j^{\binom{k}{j}} G_k^{(k)}(\bx) \left( 1 - G_{k+1}^{(k)}\big(x_1, \ldots, x_{k-1}, x_k G_k^{(k)}(\bx), x_{k+1}, \ldots, x_t\big) \right).
    \end{split}
\end{equation}
\end{lemma}
\begin{proof}
Let $B_k$ and $C_k$ be the generating functions counting the trees $\tau(\Gamma)$ ($\Gamma\in\mathcal{G}_k$) rooted at nodes of type $b$ and $c$, respectively, and $E_k$ be the generating function of those trees rooted at an undirected edge between nodes of types $b$ and $c$. 
They can also be respectively seen as the generaring functions counting $k$-connected graphs with a distinguished $(k+1)$-connected component, a distinguished $k$-clique that belongs to more than one $(k+1)$-connected components, or a distinguished $(k+1)$-connected component $C$ together with a distinguished $k$-clique in $C$ that belongs to at least another $(k+1)$-connected component $C'$.
They are specified next using the symbolic method:
\begin{align*}
    B_k(\bx) &= G_{k+1}\big(x_1, \ldots, x_{k-1}, x_k G_k^{(k)}(\bx), x_{k+1}, \ldots, x_t\big), \\
    C_k(\bx) &= \frac{1}{k!} \prod_{j\in[k]} x_j^{\binom{k}{j}} 
        \left( G_k^{(k)}(\bx) - \left( 1 + G_{k+1}^{(k)}\big(x_1, \ldots, x_{k-1}, x_k G_k^{(k)}(\bx), x_{k+1}, \ldots, x_t\big) \right)\right), \\
    E_k(\bx) &= \frac{1}{k!} \prod_{j\in[k]} x_j^{\binom{k}{j}} G_{k+1}^{(k)}\big(x_1, \ldots, x_{k-1}, x_k G_k^{(k)}(\bx), x_{k+1}, \ldots, x_t\big)
        \left( G_k^{(k)}(\bx) - 1 \right).
\end{align*}
The equation defining $B_k(\bx)$ follows directly from the decomposition discussed in the previous section, while the equation for $C_k(\bx)$ is obtained from Equation \eqref{eq:recursive} by substracting the first two terms of the exponential (because there are at least two $(k+1)$-connected components glued through the $k$-clique).
The equation for $E_k(\bx)$ can be derived by considering a $(k+1)$-connected chordal graph $\Gamma$ rooted at a $k$-clique $K$ and gluing through it a $k$-connected chordal graph $\Gamma'$ rooted at $K$, and containing at least one $(k+1)$-connected component, then further gluing some $k$-connected chordal graph to other $k$-cliques of $\Gamma'$.
The correcting factors in the last two equations are there to mark all the subcliques of the root $k$-clique and forget the order of its vertices.

Finally, recall that we consider the complete graph on $k$ vertices to be $k$-connected.
So that Proposition~\ref{prop:species} directly implies that the unrooted graphs are counted by
\begin{align*}
    G_k(\bx) = \frac{1}{k!} \prod_{j\in[k]} x_j^{\binom{k}{j}} + B_k(\bx) + C_k(\bx) - E_k(\bx).
\end{align*}
By translating this equation in light of the above three equations, one concludes the proof.
\end{proof}
%


\section{Asymptotic analysis}\label{sec:analysis}

Fix $t\ge 1$.
In this section we prove Theorems \ref{thm:enumeration} and \ref{thm:limit_law}.
We use rather classical methods from \cite{FS09}, which consist in deriving asymptotic estimates from local expansions of the generating functions from Section \ref{sec:combi} at their singularities.
Those expansions are in turn derived from successive applications of the implicit system of equations described in Lemma \ref{lem:system}, to ``transfer'' the local expansion of $G_{t+1}(\bx)$ to $G_0(x_1,1,\ldots,1)$, as illustrated by the schema in Figure \ref{fig:schema}. 

We will follow the method developed in \cite[Chapter 2]{D09}, but will need to extend some of the tools and notions present there in order to deal with multivariate generating functions and the fact that the local expansions are with respect to different variables from one step to the next.
This is the purpose of the next section.

\subsection{Proper singularity functions and singular expansions}\label{subsec:definitions_singularity}

Let $U\subset\mathbb{C}$ be an open set and $\rho: U\to\mathbb{C}$ be an analytic function.
For $u\in U$ and $\delta,\eta > 0$, a \textit{$\Delta$-domain at $\rho(u)$} is a complex region, depicted in Figure \ref{fig:delta-domain},
\begin{figure}
    \centering
    \includegraphics[width=0.4\textwidth]{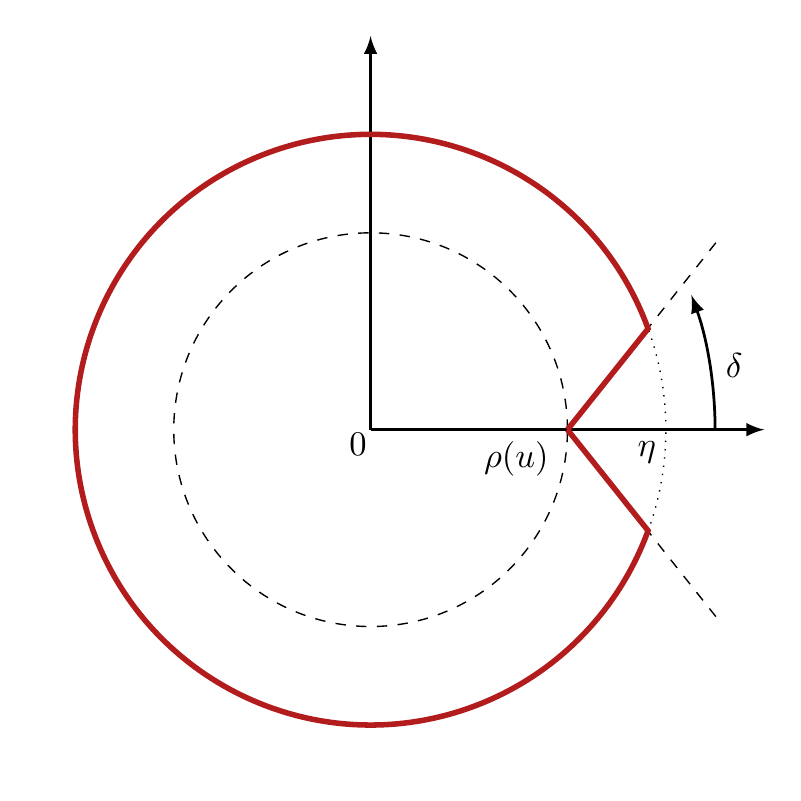}
    \caption{The open region containing 0 delimited by the closed red curve is the $\Delta$-domain $\Delta(\rho(u),\delta,\eta)$, for $\rho(u)\in\mathbb{R}_{>0}$.}
    \label{fig:delta-domain}
\end{figure}
of the form
\begin{align*}
    \Delta(\rho(u),\delta,\eta) = \Delta(\rho(u)) = \{ z\in\mathbb{C} : |z - \rho(u)| < \eta 
    \text{ and } |\arg(z/\rho(u) - 1)| > \delta \}.
\end{align*}
Our main tool is a ``transfer theorem''.
The proof can be found in \cite{D09} (see also \cite[Chapter VI.3]{FS09}).
\begin{proposition}{(Transfer Theorem \cite[Lemma 2.18]{D09}).}\label{prop:transfer}
Let $f(z,u)$ be a power series in $z$ and a parameter $u\in U$, and suppose that it admits an expansion of the form
\begin{align*}
    f(z,u) = C(u)\left( 1 - \frac{z}{\rho(u)} \right)^{-\alpha(u)}
    + O\left( \left( 1 - \frac{z}{\rho(u)} \right)^{-\beta(u)} \right),
\end{align*}
that is uniform for $u\in U$ and $z\in \Delta(\rho(u))$, and with functions $C(u)$, $\rho(u)$, $\alpha(u)$ and $\beta(u)$ that remain bounded and satisfy $\beta(u) < \Re(\alpha(u))$ for all $u\in U$.

Then the following estimate holds uniformly for $u\in U$ and as $n\to\infty$
\begin{align*}
    [z^n]f(z,u) = C(u) \frac{n^{\alpha(u) - 1}}{\Gamma(\alpha(u))} \rho(u)^{-n}
    + O\left( \rho(u)^{-n}\, n^{\max\left(\Re(\alpha(u)) - 2,\,\, \beta(u) - 1\right)} \right).
\end{align*}
\end{proposition}
\noindent By setting $u=1$ in Proposition \ref{prop:transfer}, one recovers the ``classical'' transfer theorem for univariate analytic functions, see for instance \cite[Lemma 2.15]{D09}.

Next we introduce several definitions which will allow us to extend the notion of local expansion of an analytic function at an \textit{algebraic singularity} to our multivariate setting.
First is the notion of \textit{fully movable proper singularity function}.
\begin{definition}\label{def:sing}
We say that a function $\rho(x_2,\ldots, x_t)$ is a \textbf{proper singularity function} if it satisfies the following conditions:
\begin{itemize}
    \item[(i)] It is defined in a $(t-1)$-dimensional proper complex neighbourhood of $\mathbb{R}_+^{t-1}$, where it is also analytic. 

    \item[(ii)] It is positive and real if $x_2,\ldots,x_t$ are positive and real, and it is strictly decreasing with negative derivatives in all $t-1$ positive real variables.
\end{itemize}

Furthermore we say it is \textbf{fully movable} with respect to the variables $x_2,\ldots,x_k$ if the following condition holds:
\begin{itemize}
    \item[(iii)] $\rho(x_2,\ldots, x_t) \to 0$ (resp. $\infty$) if one of the variables $x_2,\ldots,x_k$ tends to $\infty$ (resp. $0$), whereas all the other variables including $x_{k+1},\ldots, x_t$ are fixed positive real numbers.
\end{itemize}
\end{definition} 

With this notion at hand, we can define that of a \textit{positive function with a proper $\alpha$-singularity}.

\begin{definition}\label{def:alpha_sing}
Let $\alpha\in\mathbb{R}\setminus\mathbb{Z}$.
We say that a function $G(\bx)$ is a \textbf{positive function with a proper $\alpha$-singularity} if the following properties hold:
\begin{itemize}
    \item[(i)] $G(\bx)$ is a power series in $\bx = (x_1,\ldots,x_t)$ with non-negative coefficents.

    \item[(ii)] There exists a proper singularity function $\rho(x_2,\ldots,x_t)$ such that for every fixed choice of $x_2,\ldots, x_t\in\mathbb{R}_+$, $\rho(x_2,\ldots,x_t)$ is the radius of convergence of the power series $x_1 \mapsto G(\bx)$.

    \item[(iii)] For every choice of $X_0,X_1\in\mathbb{R}$ with $0 < X_0 < X_1$, there exist $\delta > 0$ and analytic functions $g_1(\bx)$, $g_2(\bx)$, that are defined and non-zero for $X_0 < |x_2|, \ldots, |x_t| < X_1$ and $|x_1 - \rho(x_2,\ldots,x_t)| < \delta$  with $x_1,\ldots,x_t$ sufficiently close to the positive real axis, such that in this range, provided that $\arg(x_1-\rho(x_2,\ldots,x_t)) \ne 0$, we have
    \begin{equation}\label{eq:Gbx_rep_def}
        G(\bx) = g_1(\bx) + g_2(\bx) \left( 1 - \frac{x_1}{\rho(x_2,\ldots,x_t)} \right)^\alpha.
    \end{equation}
    In this case we say that $x_1$ is the \textbf{leading variable} of $G(\bx)$.
\end{itemize}
\end{definition}

Finally, in order to apply Proposition \ref{prop:transfer} to a positive function with a proper $\alpha$-singularity, we need some notion of analytic continuation to a $\Delta$-domain.

\begin{definition}\label{def:aperiodic}
A positive function $G(\bx)$ with a proper $\alpha$-singularity ($\alpha\in\mathbb{R}\setminus\mathbb{Z}$) and proper singularity function $\rho(x_2,\ldots,x_t)$ is said to be \textbf{aperiodic and analytically continuable with respect to the variable $x_1$}  if the following holds:
\begin{itemize}
    \item[(i)] For every fixed choice of $x_2, \ldots, x_t\in\mathbb{R}_+$, $\rho(x_2,\ldots,x_t)$ is the unique singularity of the function $x_1\mapsto G(\bx)$ on the circle $|x_1| = \rho(x_2,\ldots,x_t)$. 

    \item[(ii)] There exists $\delta > 0$ such that $x_1\mapsto G(\bx)$ can be analytically continued to the region
    \begin{equation}\label{eq:region_analytic_continuation}
        |x_1| < |\rho(x_2,\ldots, x_t)| + \delta/2
        \quad\text{and}\quad 
        |x_1- \rho(x_2,\ldots,x_t)| > \delta.
    \end{equation}
\end{itemize}
In particular this function cannot be represented as a function of the form $x_1^a f(x_1^b)$ for some positive integers $a,b$, where $b> 1$.
\end{definition}
Fix $k\in \{2,\ldots,t\}$ and observe that setting $x_i = 1$ for $i\ne k$ in Definition~\ref{def:aperiodic}\textit{(ii)} implies that $G(x_1,1,\ldots,1,x_k,1,\ldots,1)$ can be analytically continued to a domain of the form $\Delta(x_k)$.

\subsection{Transfer properties of proper singular expansions}\label{subsec:tansfer_propeties}

We now prove certain ``transfer properties'' of proper $\alpha$-singular expansions in the neighbourhood of a proper singularity function.
Our main tools here will be the \textit{Implicit Function Theorem} (IFT) for analytic functions, and its refinement known as the \textit{Weierstrass Preparation Theorem} (WPT).
For a statement and a proof of those famous theorems, we refer the reader to \cite{KK83}.

The first property generalises \cite[Lemma 2.28]{D09} to proper singularity functions.
The proof follows the same line and we only sketch it here.
\begin{lemma}\label{lem:singularity_change}
Let $k\in\{2, \ldots, t-1\}$ and suppose that $\rho(x_2,x_3,\ldots, x_t)$ is a proper singularity function that is fully movable with respect to the  variables $x_2,\ldots, x_k$.
Then there exists a proper singularity function $\kappa(x_1,\ldots,x_{k-1},x_{k+1},\ldots, x_t)$ 
that is fully movable with respect to $x_1,\ldots,x_{k-1}$ such that
\begin{equation}\label{eq:rhokappa}
    x_1 = \rho(x_2,\ldots,x_{k-1}, \kappa(x_1,\ldots,x_{k-1},x_{k+1},\ldots, x_t), x_{k+1},\ldots, x_t).
\end{equation}

Furthermore there exists a function $K(\bx)$ that is analytic and non-zero on a $t$-dimensional complex neighbourhood of
$\mathbb{R}_+^{t}$ such that
\begin{equation}\label{eq:rhokappa2}
    x_1 - \rho(x_2,\ldots,x_t) = K(\bx)\left( x_k - \kappa(x_1,\ldots,x_{k-1},x_{k+1},\ldots, x_t) \right).
\end{equation}
\end{lemma}
\begin{proof}
Suppose first that $x_1,\ldots, x_t$ are positive real variables.
Since $\rho$ is strictly decreasing and tends to $0$ (resp. $\infty$), if one of the variables $x_2,\ldots, x_k$ tends to $\infty$ (resp. $0$) then it immediately follows from the continuity of $\rho$ that a function $\kappa = \kappa(x_1,\ldots,x_{k-1},x_{k+1},\ldots, x_t)$ satisfying \eqref{eq:rhokappa} exists. 
Furthermore, $\kappa$ is strictly decreasing and tends to $0$ (resp. $\infty$) if one of the variables $x_1,\ldots,x_{k-1}$ tends to $\infty$ (resp. $0$). 

Next, fix $x_2, \ldots, x_t \in\mathbb{R}_+$ and set $x_{1} = \rho(x_{2},\ldots, x_{t})$. 
Since from Definition \ref{def:sing} $\rho$ is analytic and satisfies $\frac{\partial}{\partial x_k} \rho < 0$ for $2\le k\le t$, it follows, by applying the IFT to \eqref{eq:rhokappa}, that the function $\kappa$ can be (uniquely) analytically continued to a complex neighbourhood of $(x_{1},\ldots, x_{k-1}, x_{k+1},\ldots, x_{t})$.
In fact, using the WPT in the degree one case, it can further be shown that there exists a function $K(\bx)$ that is analytic and non-zero in a complex neighbourhood of $\bx$ such that \eqref{eq:rhokappa2} holds. 

Finally, a standard analytic continuation argument shows that both $\kappa$ and $K$ can be {\it globally} defined so that \eqref{eq:rhokappa2} holds in the proposed range, that is, a complex neighbourhood of $\mathbb{R}_+^{t}$.
\end{proof}

An important consequence of Lemma~\ref{lem:singularity_change} is that for $k\in [t]$ the representation \eqref{eq:Gbx_rep_def} can be rewritten into
\begin{equation*}
    G(\bx) = g_1(\bx) + \overline g_2(\bx) \left( 1 - \frac {x_k}{\kappa} \right)^\alpha,
\end{equation*}
with $\kappa = \kappa(x_1,\ldots,x_{k-1},x_{k+1},\ldots, x_t)$ and where the analytic function
\begin{equation*}
    \overline g_2(\bx) = g_2(\bx) \left( \frac{K(\bx)\kappa}{\rho(x_2,\ldots,x_t)} \right)^{\alpha}
\end{equation*}
is defined and non-zero for $X_0 < |x_2|, \ldots, |x_t| < X_1$.
This means that any of the variables $x_1,\ldots, x_k$ can be the leading one in the definition of a positive function with a proper $\alpha$-singularity, provided that the proper singularity function $\rho(x_2,\ldots, x_t)$ is fully movable with respect to $x_2,\ldots, x_k$.
Furthermore $\kappa$ is certainly a singularity of the mapping $x_k \mapsto G(\bx)$. 
And by the monotonicity property in Definition~\ref{def:sing}\textit{(ii)} there is no smaller positive real singularity. 
Thus, $\kappa$ is the radius of convergence of the mapping $x_k \mapsto G(\bx)$, provided that $x_1,\ldots,x_{k-1},x_{k+1},\ldots, x_t\in\mathbb{R}_+$.

Next, we extend \cite[Lemma~2.27]{D09} to the context of positive functions with a proper $\alpha$-singularity.
\begin{lemma}\label{lem:diff_int}
For $k \in [t]$ and $\alpha\in\mathbb{R}\setminus\mathbb{Z}$, let $G(\bx)$ be a positive function with a proper $\alpha$-singularity, aperiodic and analytically continuable with respect to $x_1$,  and with a proper singularity function $\rho(x_2,\ldots, x_t)$ that is fully movable with respect to $x_2,\ldots,x_k$. 
Set
\begin{equation}\label{eq:H_integral}
    D(\bx) = \frac{\partial}{\partial x_k} G(\bx)
    \quad\text{and}\quad
    H(\bx) = \int_0^{x_k} G(x_1,\ldots,x_{k-1},y,x_{k+1}, \ldots,x_t)\, dy.
\end{equation}
Then $D(\bx)$ is a positive function with a proper $(\alpha-1)$-singularity while $H(\bx)$ is a positive function with a proper $(\alpha+1)$-singularity, and both are aperiodic and analytically continuable with respect to $x_1$.
Furthermore the proper singularity functions of $G(\bx)$, $D(\bx)$ and $H(\bx)$ coincide.
\end{lemma}
\begin{proof}
Fix $\delta > 0$.
First, the analytic continuability of both mappings $x_1 \mapsto D(\bx)$ and $x_1 \mapsto H(\bx)$ to 
a region of the form \eqref{eq:region_analytic_continuation} is immediate by assumption on $G(\bx)$ and properties of the derivative and the integral.
Second, if $|x_1- \rho(x_2,\ldots,x_t)| < \delta$ then Lemma~\ref{lem:singularity_change} implies that there exist $\delta' > 0$ and a proper singularity function $\kappa = \kappa(x_1,\ldots,x_{k-1},x_{k+1},\ldots, x_t)$ such that for $|x_k - \kappa| < \delta'$, $G(\bx)$ can be represented as
\begin{equation}\label{eq:representation_G_kappa}
    G(\bx) = g_1(\bx) + \overline g_2(\bx) \left( 1 - \frac {x_k}{\kappa} \right)^{\alpha}.
\end{equation}

Set $d_1(\bx) = (\partial/\partial x_k) g_1(\bx)$ and $d_2(\bx) = \overline g_2(\bx)/(2\kappa) + (\partial/\partial x_k) \overline g_2(\bx) \left( 1 - x_k/\kappa \right)$.
Then taking the partial derivative of \eqref{eq:representation_G_kappa} with respect to $x_k$ gives
\begin{align*}
    D(\bx) 
    = \frac{\partial}{\partial x_k} g_1(\bx) + \frac{\partial}{\partial x_k} \overline g_2(\bx) \left( 1 - \frac{x_k}{\kappa} \right)^{\alpha} + \frac{\overline g_2(\bx)}{2\kappa} \left( 1 - \frac{x_k}{\kappa} \right)^{\alpha - 1}
    = d_1(\bx) + d_2(\bx) \left( 1 - \frac{x_k}{\kappa} \right)^{\alpha - 1}.
\end{align*}

Now, in order to compute the integral of \eqref{eq:representation_G_kappa} we first compute the Taylor expansions of the functions $g_1(\bx)$ and $\overline g_2(\bx)$ at $x_k\sim \kappa$ and obtain a representation of the form
\begin{equation}\label{eq:G_series}
    G(\bx) = \sum_{\ell \ge 0} G_\ell(x_1,\ldots,x_{k-1},x_{k+1},\ldots, x_t)\left(1 - \frac{x_k}{\kappa}\right)^{\alpha\ell}
\end{equation}
that is certainly convergent for $|x_k - \kappa| < \delta'$. Next we split up the 
integral in \eqref{eq:H_integral} into three parts
\begin{align*}
    I_1(x_1, \dots,x_{k-1},x_{k+1},\dots,x_t) &:= \int_0^{(1-\eta) \kappa} G(x_1,\ldots,x_{k-1},y,x_{k+1}, \ldots,x_t)\, dy, \\
    I_2(x_1, \dots,x_{k-1},x_{k+1},\dots,x_t) &:= \int_{(1-\eta) \kappa}^{\kappa} G(x_1,\ldots,x_{k-1},y,x_{k+1}, \ldots,x_t)\, dy, \\
    I_3(\bx) &:= \int_{\kappa}^{x_k} G(x_1,\ldots,x_{k-1},y,x_{k+1}, \ldots,x_t)\, dy,
\end{align*}
where $\eta>0$ is chosen in such a way that $\eta  < \delta'/ |\kappa|$. 
The first integral is certainly an analytic function in $x_1,\ldots,x_{k-1},x_k,\ldots, x_t$, as a definite integral with respect to $y$ in a range where $G$ is analytic. 
The second integral can be directly computed by the series expansion \eqref{eq:G_series}
\begin{align*}
    I_2(x_1,\dots,x_{k-1},x_{k+1},\dots,x_t) 
    & = \int\limits_{(1 - \eta)\kappa}^{\kappa} G(x_1,\ldots,x_{k-1},y,x_{k+1},\ldots,x_t)\, dy \\
    & = \sum_{\ell\ge 0} G_\ell(x_1,\ldots,x_{k-1},x_{k+1},\ldots,x_t) 
    \int\limits_{(1 - \eta) \kappa}^{\kappa}(1 - y/\kappa)^{\alpha\ell}\, dy  \\
    & = \kappa \sum_{\ell\ge 0} \frac{G_\ell(x_1,\ldots,x_{k-1},x_{k+1},\ldots,x_t)}{\alpha\ell + 1} \eta^{\alpha\ell + 1}.
\end{align*}
This series is absolutely convergent and represents an analytic function in $x_1,\ldots,x_{k-1}$, $x_{k+1},\ldots,x_t$. 
Finally, for the third integral we use again the series expansion (\ref{eq:G_series})
and obtain
\begin{align*}
    I_3(\bx) 
    & = \int_{\kappa}^{x_k} G(x_1,\ldots,x_{k-1},y,x_{k+1},\ldots,x_t)\, dy \\
    & = \sum_{\ell\ge 0} G_\ell(x_1,\ldots,x_{k-1},x_{k+1},\ldots,x_t) 
    \int\limits_{\kappa}^{x_k}(1 - y/\kappa)^{\alpha\ell}\, dy \\
    & = - \kappa \sum_{\ell\ge 0} 
    \frac{G_\ell(x_1,\ldots,x_{k-1},x_{k+1},\ldots,x_t)}{\alpha\ell + 1}(1 - x_k/\kappa)^{\alpha\ell + 1}.
\end{align*}
This series representation can be rewritten into
\begin{align*}
    I_3(\bx) = h_1(\bx) + h_2(\bx) \left( 1 - \frac{x_k}{\kappa} \right)^{\alpha + 1}.
\end{align*}
Next, note that since
\begin{align*}
    \overline g_2(x_1,\ldots,x_{k-1},\kappa,x_{k+1},\ldots,x_t) 
    = -\frac{\alpha + 1}{\kappa} h_2(x_1,\ldots,x_{k-1},\kappa,x_{k+1},\ldots,x_t),
\end{align*}
then both $\overline g_2$ and $h_2$ are non-zero, even if $|x_k - \kappa| < \delta''$ for a sufficiently small $\delta'' > 0$.
Moreover, since the coefficients of $G(\bx)$ and $H(\bx)$ are non-negative we have $g_1(\bx) > 0$ for $x_j\in\mathbb{R}_+$ and $\overline h_1(\bx) = h_1(\bx) + I_1(x_1,\dots,x_{k-1},x_{k+1},\dots,x_t) + I_2(x_1,\dots,x_{k-1},x_{k+1},\dots,x_t) > 0$.

Finally, by another application of Lemma~\ref{lem:singularity_change}, we can rewrite $D(\bx)$ and $H(\bx)$ into 
\begin{align*}
    D(\bx) = d_1(\bx) + \overline d_2(\bx) \left( 1 - \frac{x_{1}}{\rho(x_2,\ldots,x_t)} \right)^{\alpha - 1}
    \text{and }
    H(\bx) = \overline h_1(\bx) + \overline h_2(\bx) \left( 1 - \frac{x_{1}}{\rho(x_2,\ldots,x_t)} \right)^{\alpha + 1}
\end{align*}
with $\overline d_2,\overline h_2\ne 0$. 
This completes the proof.
\end{proof}

Finally, we generalise \cite[Theorem~2.21]{D09}.
This theorem states that if $G(x,u)$ is a univariate function, with parameter $u=(x_2,\ldots,x_t)$, defined implicitely in terms of another function $F(x,u,y)$ (i.e. such that $F(x,u,G(x,u))=0$) that admits a $1/2$-singular expansion at some $R>0$, then $G(x,u)$ also admits a $1/2$-singular expansion at some $\rho < R$.
The next result extends this to the case where $F$ has a \textbf{proper} $1/2$-singularity.
The proof follows the same lines and we sketch it next.

\begin{lemma}\label{lem:implicit}
Suppose that $F(\bx,y)$ is a positive function in $t+1$ variables with a proper $1/2$-singularity and
singularity function $R(x_2,\ldots,x_t;y)$ that is fully movable with respect to the variables $x_2,\ldots, x_k$ and $y$
for some $2\le k \le t$. Furthermore assume that $F(x_1,x_2,\ldots,x_t,y) = 0$
if one of the variables $x_1,\ldots,x_k$ are zero. 
Then the functional equation
\begin{equation}\label{eq:implict_lemma}
G =  \exp(F(\bx,G))
\end{equation}
has a unique solution $G = G(\bx)$ with $G({\bf 0}) = 1$ which is a positive function with a proper $1/2$-singularity, too.
Its singularity function $\rho(x_2,\ldots,x_t)$ is fully movable with respect to 
the variables  $x_2,\ldots, x_k$ and satisfies
\begin{align*}
    \rho(x_2,\ldots, x_t) < R\big(x_2,\ldots,x_t; G(\rho(x_2,\ldots, x_t),x_2,\ldots, x_t)\big).
\end{align*}
Moreover, if $F(\bx,y)$ is periodic and analytically continuable with
respect to the variable $x_1$ then the same property holds for $G(\bx)$.
\end{lemma}
\begin{proof}
First, by iteration (or by the IFT), Equation \eqref{eq:implict_lemma} admits a unique power series solution with $G({\bf 0}) = 1$ and non-negative coefficients.

Next we define a singularity function $\rho(x_2,\ldots,x_t)$. 
For this purpose we fix $x_2,\ldots,x_t\in\mathbb{R}_+$ and vary $x_1$. 
We claim that there exists a unique $\overline x_1 > 0$ such that for $\overline{\bx} = (\overline x_1, x_2,\ldots, x_t)$ we have $\overline x_1 <  R(x_2,\ldots,x_t, G(\overline{\bx}))$
and satisfying
\begin{equation}\label{eq:implicit_sing}
    G(\overline{\bx}) = \exp(F(\overline{\bx},G(\overline{\bx}))),
\end{equation}
and 
\begin{equation}\label{eq:branch_point}
    1 = \exp(F(\overline{\bx},G(\overline{\bx})))\frac{\partial F}{\partial y} (\overline{\bx},G(\overline{\bx})).
\end{equation}
Since all the coefficients of $G$ are non-negative, the solution function $G$ is strictly increasing in $x_1$. 
Consequently, the factor $\exp(F(\overline{\bx},G(\overline{\bx})))$ in \eqref{eq:branch_point} is also strictly increasing in $x_1$.

Now we study the factor $(\partial F/\partial y)(\overline{\bx},G(\overline{\bx}))$ in \eqref{eq:branch_point}. By assumption we have 
\begin{align*}
    \frac{\partial F}{\partial y}(0,x_2,\ldots,x_t,G(0, x_2,\ldots,x_t)) = 0.
\end{align*}
Moreover, since $F$ is a positive function with a proper $1/2$-singularity, by Lemma \ref{lem:diff_int} $\partial F/\partial y$ is a positive function with a proper $(-1/2)$-singularity, that is, when $x_1\sim R(x_2,\ldots,x_t, y)$ it can be represented as
\begin{align*}
    \frac{\partial F}{\partial y} (\bx,y) = \overline f_1(\bx,y) + \overline f_2(\bx,y) \left( 1 - \frac{x_1}{R(x_2,\ldots,x_t,y)} \right)^{-1/2},
\end{align*}
where $\overline f_1$ and $\overline f_2$ are analytic and $\overline f_2 \ne 0$. 
Since $G$ is strictly increasing in $x_1$ and $R$ is a proper singularity function fully movable in $y$, $R(x_2,\ldots,x_t,G(\bx))$ is strictly decreasing in $x_1$ and goes to $0$ as $x_1\to \infty$. 
Therefore, $\left(1 - x_1/R(x_2,\ldots,x_t,G(\bx)) \right)^{-1/2}$ is strictly increasing in $x_1$ and unbounded while $x_1<R(x_2,\ldots,x_t,G(\bx))$. 
The same is true for $(\partial F/\partial y)(\bx,G(\bx))$ because $\overline f_1(\bx,G(\bx))$ and $\overline f_2(\bx,G(\bx))$ are strictly increasing functions in $x_1$ when $x_1<R(x_2,\ldots,x_t,G(\bx))$. 
Our claim follows.

From \cite[Theorem~2.19]{D09}, which amounts to evaluating the parameter $u$ in $\mathbb{R}_+$ in \cite[Theorem~2.21]{D09}, this implies that for $x_2,\ldots,x_t\in\mathbb{R}_+$ the univariate function $x_1 \to G(\bx)$ has a $1/2$-singularity at $x_1= \overline x_1$. 
Therefore we set
\begin{align*}
    \rho(x_2,\ldots,x_t) := \overline x_1.
\end{align*}

The system formed by equations \eqref{eq:implicit_sing} and \eqref{eq:branch_point} can be used to get more information on $\rho$. 
Notice first that, given $x_2,\ldots, x_t$, the system determines 
$\overline x_1 = \rho(x_2,\ldots,x_t)$ and $G(\overline \bx)$.
Then, since the determinant 
\begin{align*}
    \left| 
    \begin{array}{cc}  
        - e^F\ds\frac{\partial F}{\partial x_1}  & 1 - e^F\ds\frac{\partial F}{\partial y} \\
        - e^F\ds\frac{\partial F}{\partial x_1}\frac{\partial F}{\partial y}   -e^F\ds\frac{\partial^2 F}{\partial y \partial x_1} & 
        - e^F \left( \ds\frac{\partial F}{\partial y} \right)^2  - e^F \ds\frac{\partial^2 F}{\partial y^2}  
    \end{array} 
    \right| 
    &= e^{2F} \left| 
    \begin{array}{cc}  
       \ds \frac{\partial F}{\partial x_1} & 0 \\
       \ds \frac{\partial F}{\partial x_1}\frac{\partial F}{\partial y} + \frac{\partial^2 F}{\partial y \partial x_1} & \left( \ds\frac{\partial F}{\partial y} \right)^2 + \ds\frac{\partial^2 F}{\partial y^2}  
    \end{array} 
    \right| \\
    &= e^{2F} \ds\frac{\partial F}{\partial x_1} \left( \left( \frac{\partial F}{\partial y} \right)^2 + \ds\frac{\partial^2 F}{\partial y^2} \right) 
\end{align*}
is positive, it follows by the IFT that the function $\rho(x_2,\ldots,x_t)$ can be locally analytically continued.
Fix now some $2\le j \le k$. 
By differentiating \eqref{eq:implicit_sing} with respect to $x_j$, one obtains
\begin{align*}
    0 
    &= \frac{\partial [G(\overline \bx)]}{\partial x_j} - \exp(F(\overline \bx, G(\overline \bx))) \left[\frac{\partial F}{\partial x_1} (\overline \bx, G(\overline \bx)) \frac{\partial \rho}{\partial x_j} + \frac{\partial F}{\partial x_j}(\overline \bx, G(\overline \bx)) + \frac{\partial F}{\partial y}(\overline \bx, G(\overline \bx)) \frac{\partial [G(\overline \bx)]}{\partial x_j} \right] \\
    &= - \exp(F(\overline \bx, G(\overline \bx))) \left[\frac{\partial F}{\partial x_1} (\overline \bx, G(\overline \bx)) \frac{\partial \rho}{\partial x_j} + \frac{\partial F}{\partial x_j}(\overline \bx, G(\overline \bx))\right].
\end{align*}
In other words,
\begin{align*}
    \ds\frac{\partial \rho}{\partial x_j} = - \frac{ \ds\frac{\partial F}{\partial x_j}(\overline \bx, G(\overline \bx))}{ \ds\frac{\partial F}{\partial x_1}(\overline \bx, G(\overline \bx))} < 0.
\end{align*}
This means that $\rho$ is strictly decreasing in all variables, provided they are real and positive.
Let us finally consider the behaviour of $\rho$ as $x_j$ tends to $0$ or $+\infty$.
Suppose first that $\overline x_1 = \rho$ is bounded away from 0 when $x_j\to\infty$. 
Then $G(\overline \bx) \to +\infty$, and $R\to 0$. 
However, this is impossible since $\overline x_1 < R$.
Thus $\rho\to 0$ as $x_j\to\infty$.
On the other hand, suppose that $\overline x_1 = \rho$ stays bounded when $x_j\to 0$. 
In this case, $G(\overline \bx)$ stays bounded and so by assumptions on the zeros of $F$, $F(\overline x, G(\overline \bx))) \to 0$ and $(\partial F/\partial y)(\overline x,G(\overline \bx))) \to 0$ as $x_j \to 0$. 
However, this is not possible by \eqref{eq:branch_point}.
Summing up, this means from Definition \ref{def:sing} that the function $\rho$ is a proper singularity function that is fully movable with respect
to $x_2,\ldots, x_k$.

Furthermore, it follows from \cite[Theorem~2.21]{D09} that we also get an expansion of the form (\ref{eq:Gbx_rep_def}) with $\alpha = 1/2$ for $G(\bx)$.
And it remains to check that for $x_2,\ldots,x_t\in\mathbb{R}_+$, the mapping $x_1\mapsto G(\bx)$ admits an analytic continuation away from $\rho$, in a region of the form \eqref{eq:region_analytic_continuation}. 
To that end, we now consider \eqref{eq:implict_lemma} as a functional equation for the function $x_1\mapsto G(\bx)$. 
Since $G(\bx)$ can be written as $G(\bx) = 1 + \tilde G(\bx)$, where $\tilde G$ is a power series with non-negative coefficients, we have that
\begin{align*}
    G(\bx) = \exp(F(\bx, G(\bx))) 
    &= \exp(F(\bx, 1 + \tilde G(\bx))) \\
    &= \exp(F(\bx, 1) + \tilde F(\bx, \tilde G(\bx))) \\
    &= 1 + F(\bx, 1) + \tilde F(\bx, \tilde G(\bx)) + \sum_{n\geq 2} \frac{(F(\bx, 1) + \tilde F(\bx, \tilde G(\bx)))^n}{n!},
\end{align*}
where $\tilde F(\bx,y)$ is a power series with non-negative coefficients. 
Now, since $F(\bx,y)$ is aperiodic in $x_1$, it follows that $G(\bx)$ has to also be aperiodic with respect to $x_1$. 
This implies that $\rho(x_2,\ldots,x_t)$ is the unique dominant singularity of the function $x_1\mapsto G(\bx)$ and we conclude by a standard compactness argument that it can be analytically continued to a region of the form \eqref{eq:region_analytic_continuation}.
\end{proof}

\subsection{Proofs of the main results}\label{subsec:proofs}

Fix  $t\ge 1$ and recall that starting with $G_{t+1}$, which is an explicit monomial, one can recursively obtain the generating functions $G_t, G_{t-1}, \ldots, G_1$, and finally $G_0 = \exp(G_1)$.
Let us next discuss the first step of this induction, from $G_{t+1}$ to $G_t$, since it is slightly different from the general step.

\begin{proposition}\label{prop:init_induction}
Let $t\ge 1$, $x_2,\ldots, x_t\in\mathbb{R}_+$.
Then there exist two functions $h_1(x_1)$ and $h_2(x_1)$, that are analytic and non-zero at $x_1 = 1/e$, such that for $x_1$ in a neighbourhood of $1/e$ we have
\begin{equation}\label{eq:Gt_rep}
    G_t(\bx) = \frac{\prod_{j =1}^t  x_j^{\binom{t}{j}} } {t!} \left( h_1(tX) + h_2(tX) (1 - etX)^{3/2} \right),
    \qquad\text{where } X = \prod_{j =1}^t  x_j^{\binom{t}{j-1}}.
\end{equation}
\end{proposition}

\begin{proof}
From Equation \eqref{eq:G_t+1} and by \eqref{eq:rooting} we directly get
\begin{equation*}
    G_{t+1}^{(t)}(\bx) = \prod_{j =1}^t  x_j^{\binom{t}{j-1}}.
\end{equation*}
Consequently, from the relation \eqref{eq:recursive} the function $G_t^{(t)} = G_t^{(t)}(\bx)$ satisfies the equation
\begin{equation*}
    G_t^{(t)} = \exp\left( \prod_{j =1}^t  x_j^{\binom{t}{j-1}}    [G_t^{(t)}]^t \right).
\end{equation*}

Let $T(z)$ denote the \textit{tree function}, i.e. that satisfies the equation $T(z) = z\exp(T(z))$.
Then, using the change of variable $z = tX$ with $X = \prod_{j =1}^t  x_j^{\binom{t}{j-1}}$, we can represent $G_t^{(t)}(\bx)$ as 
\begin{equation*}
    G_t^{(t)}(\bx) = \left( \frac{T(tX)}{tX} \right)^{1/t} = \exp\left( T(tX)/t \right).
\end{equation*}
With the help of (\ref{eq:dissymetry}) and the relation $T(z) = z\exp(T(z))$, this also leads to
\begin{align}
    G_t(\bx) &= \frac{1}{t!} \prod_{j=1}^t x_j^{\binom{t}{j}} G_t^{(t)}(\bx) 
    - \frac t{(t + 1)!} \prod_{j=1}^t x_j^{\binom{t + 1}{j}} \left[ G_t^{(t)}(\bx) \right]^{t+1}  \nonumber \\
    &= \frac{\prod_{j=1}^t x_j^{\binom{t}{j}}}{t!} \exp\left( \frac{T(tX)}{t} \right) \left( 1 - \frac{T(tX)}{t + 1} \right).
    \label{eqGtrep}
\end{align}
It is well known (see for instance \cite{FS09}) that $T(z)$ has its dominant singularity at $z_0 = 1/e$ and a local 
Puiseux expansion, when $z$ is in a neigbourhood of $z_0$, of the form
\begin{equation*}
    T(z) = 1 - \sqrt{2} \sqrt{1- ez} + \frac 23 (1-ez) - \frac{11\sqrt 2}{36} (1-ez)^{3/2} + O\left( (1 - ez)^2 \right).
\end{equation*}
Furthermore, $z_0 = 1/e$ is the only singularity on the circle $|z| = 1/e$ and $T(z)$ can be analytically
continued to a region of the form $|z| < 1/e + \delta/2$, $|z - 1/e| > \delta$ for some $\delta > 0$.

Hence, setting $z=tX$, we get 
\begin{align*}
    \exp\left( \frac{T(z)}t \right) \left( 1 - \frac{T(z)}{t + 1} \right) 
    &= \frac{te^{1/t}}{t + 1} \left( 1 - \frac{1}{t^2} (1 - ex_1) + \frac{2\sqrt{2}(t + 1)}{3t^3} (1 - ex_1)^{3/2} + O\left( (1 - ex_1)^2 \right) \right) \\
    &= h_1(x_1) + h_2(x_1) (1 - ex_1)^{3/2},
\end{align*}
where $h_1(x_1)$ and $h_2(x_1)$ are functions that are analytic and non-zero for $x_1$ in a neighbourhood of $1/e$.
This directly leads to the claimed local representation of $G_t(\bx)$.
\end{proof}

Note that the appearance of the dominant singularity $(1 - etX)^{3/2}$ is not unexpected since $G_t^{(t)}(\bx)$ has a dominant singularity of the form $\sqrt{1 - etX}$ and $G_t(\bx)$ and is as per \eqref{eq:integrating} -- more or less -- the integral of $G_t^{(t)}(\bx)$.

\medskip

Furthermore, one can deduce from Proposition~\ref{prop:init_induction} the case $k=t$ of Theorem~\ref{thm:enumeration} by a direct application of Proposition \ref{prop:transfer} (setting $u=1$).
Similarly, a central limit theorem for the case $k=t$ of Theorem~\ref{thm:limit_law} follows immediately from Proposition~\ref{prop:init_induction} by an application of \cite[Theorem~2.25]{D09}.

\medskip

For $k< t$, Theorems~\ref{thm:enumeration} and \ref{thm:limit_law} can also be deduced from local representations of a form similar to \eqref{eq:Gt_rep}, modulo some technical conditions.
Thus the main step of the proofs is to show that the above representation for $G_t(\bx)$ implies corresponding representation for $G_{t-1}(\bx), G_{t-2}(\bx), \ldots, G_1(\bx)$.
This is the object of the next proposition.
Before stating it, let us remark the following.
\begin{remark}\label{rem:zeros}
For $k\in[t+1]$, the function $G_k(\bx)$ admits $\prod_{j\in[k]} x_j^{\binom{k}{j}}$ as a factor.
If $k = t+1$, this is Equation \eqref{eq:G_t+1}.
For $k\leq t$, it follows from Equation \eqref{eq:integrating}.
And by \eqref{eq:rooting} this implies that the function $G_{k}^{(k-1)}(\bx)$ has $\prod_{j\in[k]} x_j^{\binom{k-1}{j-1}}$ as a factor.
In particular, $G_{k}^{(k-1)}(\bx)$ is zero if one of the variables $x_1,\ldots, x_k$ are zero.
\end{remark}

\begin{proposition}\label{prop:inductive_step}
Suppose that $2 \le k \le t$ and let $G_k(\bf x)$ be a positive function with a proper $3/2$-singularity that is aperiodic and analytically continuable with respect to $x_1$, and with a proper singularity function $\rho_k(x_2,\ldots,x_t)$ that is fully movable with respect to $x_2,\ldots, x_k$.

Then the function $G_{k-1}(\bf x)$ is also a positive function with a proper $3/2$-singularity that is aperiodic and analytically continuable with respect to $x_1$, where the singularity function $\rho_{k-1}(x_2,\ldots,x_t)$ 
is again fully movable with respect to $x_2\ldots, x_k$.    
Moreover, if $x_2,\ldots, x_t\in\mathbb{R}_+$ then 
\begin{equation}\label{eq:prop_inductive_step}
    \rho_{k-1}(x_2,\ldots, x_t) < \rho_k(x_2,\ldots, x_t).
\end{equation}
\end{proposition}

\begin{proof}
In a first step, using Lemma~\ref{lem:singularity_change} we replace $\rho_k(x_2,\ldots,x_t)$ by the proper singularity function $\kappa_k = \kappa_k(x_1,\ldots,x_{k-2},x_{k},\ldots, x_t)$, that is fully movable with respect to $x_1,\ldots,x_{k-2}$, so that we can represent $G_k(\bx)$ as
\begin{align*}
    G_k(\bx) = \overline g_1(\bx) + \overline g_2(\bx) 
    \left( 1- \frac{x_{k-1}}{\kappa_k} \right)^{3/2}.
\end{align*}

Next, we deduce from Lemma~\ref{lem:diff_int} and the relation \eqref{eq:rooting} that $G_{k}^{(k-1)}(\bx)$ is a positive function with a proper $1/2$-singularity and admits the same proper singularity function $\kappa_k$ as $G_k(\bx)$, in particular it is fully movable with respect to $x_1,\ldots,x_{k-2}$. 

From there we apply Lemma~\ref{lem:implicit} to the relation \eqref{eq:recursive}, noting that
\begin{align*}
    F(\bx,y) = G_{k}^{(k-1)}(x_1,\ldots, x_{k-2}, x_{k-1}y, x_k,\ldots, x_t)
\end{align*}
is a positive function with a proper $1/2$-singularity and that by Remark \ref{rem:zeros} $F(\bx,y)$ has zeros $x_1,\ldots, x_k$.
Furthermore, it admits a proper singularity function given by
\begin{align*}
    R(x_1,\ldots,x_{k-2},x_k,\ldots, x_t,y) = \frac 1y \kappa_k(x_1,\ldots,x_{k-2},x_k,\ldots, x_t).
\end{align*}
Clearly $R$ is fully movable in $x_1,\ldots, x_{k-2},x_k$ and $y$.
Consequently, using Lemma \ref{lem:implicit} the solution function $y=G_{k-1}^{(k-1)}(\bx)$ is a positive function with a proper $1/2$-singularity, and leading variable $x_{k-1}$, for which the singularity function $\kappa_{k-1} = \kappa_{k-1}(x_1,\ldots,x_{k-2},x_k,\ldots, x_t)$ satisfies
\begin{align*}
    \kappa_{k-1}(x_1,\ldots,x_{k-2},x_k,\ldots, x_t) < \frac{\kappa_k(x_1,\ldots,x_{k-2},x_k,\ldots, x_t)}{G_{k-1}^{(k-1)}(\bx)}.
\end{align*}
Note that $G_{k-1}^{(k-1)}({\bf 0}) = 1$. 
Hence, $G_{k-1}^{(k-1)}(\bx) > 1$, and it follows that
\begin{align*}
    \kappa_{k-1}(x_1,\ldots,x_{k-2},x_k,\ldots, x_t) < \kappa_k(x_1,\ldots,x_{k-2},x_k,\ldots, x_t).
\end{align*}

Finally, we apply Lemma~\ref{lem:diff_int} on relation \eqref{eq:integrating} and obtain that $G_{k-1}(\bx)$ is a positive function with a proper $3/2$-singularity and leading variable $x_{k-1}$. 
By another application of Lemma~\ref{lem:singularity_change} we see that we can change it back to the leading variable $x_1$, such that the corresponding proper singularity function $\rho_{k-1}(x_2,\ldots,x_k)$ satisfies \eqref{eq:prop_inductive_step}.
\end{proof}

We are now in a position to prove the two main results of the paper.

\paragraph{Proof of Theorem \ref{thm:enumeration}.}

Proposition~\ref{prop:init_induction} implies that $G_t(\bx)$ is a positive function with a proper $3/2$-singularity, and is aperiodic and analytically continuable with respect to $x_1$.
In particular, compare \eqref{eq:Gbx_rep_def} with \eqref{eq:Gt_rep} and note that $x_1$ appears in $X$ only in the first power $x_1^{\binom{t}{0}}$.
In this case, the proper singularity function is explicitly given by
\begin{equation*}
    \rho_t(x_2,\ldots,x_t) = \frac 1{et} \prod_{j=2}^t  x_j^{-\binom{t}{j-1}}.
\end{equation*}

From there, successive applications of Proposition~\ref{prop:inductive_step} imply that the function $G_k(\bx)$ also has these properties for each $k\in [t]$. 
Since the exponential is an entire function this also holds for $G(\bx) = G_0(\bx) = \exp(G_1(\bx))$.
And we conclude the proof by setting $x_2 = \cdots = x_t = 1$ then applying Proposition \ref{prop:transfer}. \qed

\paragraph{Proof of Theorem \ref{thm:limit_law}.}

Suppose that $(x_2,\ldots,x_t)$ is in a sufficiently small complex neighbourhood $U$ of $(1,\ldots, 1)$ in $\mathbb{C}^{t-1}$.
From Proposition \ref{prop:init_induction} and $k-1$ successive applications of Proposition~\ref{prop:inductive_step}, we derive a local representation of the form \eqref{eq:Gbx_rep_def} holds for $G(\bx) = G_k(\bx)$, with $\rho(x_2,\ldots,x_k) = \rho_k(x_2,\ldots,x_k)$ and $\alpha=3/2$, when $x_1$ is close to $\rho_k(x_2,\ldots,x_k)$.
Furthermore, by continuity there exists $\delta> 0$ such that the function $x_1 \mapsto G_k(\bx)$ is still analytically continuable to a region of the form \eqref{eq:region_analytic_continuation}, with $\rho = \rho_k$.
And we deduce from Proposition \ref{prop:transfer} the following asymptotic estimate for the coefficients of $x_1$ in $G_k(\bx)$
\begin{equation*}
    [x_1^n]\, G_k(\bx) = C_k(x_2,\ldots, x_k) \, n^{-5/2} \, \rho_k(x_2,\ldots,x_t)^{-n} \, (1 + o(1))
    \qquad\text{as } n\to\infty,
\end{equation*}
for some non-zero function $C_k(x_2,\ldots, x_k)$ analytic in $U$.
This leads to a \textit{quasi-power} situation for the probability generating function 
\begin{equation*}
    \mathbb{E}\left[ x_2^{X_2}\cdots x_t^{X_t} \right] 
    = \frac{ [x_1^n]\, G_k(\bx) } {[x_1^n]\, G_k(x_1,1,\ldots,1) } 
    \sim  \frac{ C_k(x_2,\ldots, x_k)} {C_k(1,\ldots, 1) }  \left( \frac {\rho_k(1,\ldots,1) } { \rho_k(x_2,\ldots,x_t) } \right)^n.
\end{equation*}

Finally, setting $x_j = e^{u_j}$ and $\lambda_n = n$ in \cite[Theorem~2.22]{D09} implies the claimed joint central limit theorem.
Note that one could alternatively apply \cite[Theorem~2.25]{D09}.
Furthermore, the relation $G_0(\bx) = \exp\left(G_1(\bx)\right)$ implies, as above, that $G(\bx) = G_0(\bx)$ has the same singularities and singular expansion as $G_1(\bx)$, up to a multiplicative constant.
This concludes the proof.  \qed

%
%
\section{Concluding remarks}\label{sec:conclusion}

With the help of a computer algebra system, making use of the representation in Lemma \ref{lem:combinatorial_integration}, we have been able to compute the following table of numerical values for the singularities $\rho_{t,k}$. 
We have stopped at $t=7$ since the size of the system of functional equations needed to determine $\rho_{t,k}$ grows too fast.
\begin{table}[h]
\centering
\begin{tabular}{c|ccccccc}
    \toprule
    & $k=1$ & $k=2$ & $k=3$ & $k=4$ & $k=5$ & $k=6$ & $k=7$\\
    \midrule
    $t=1$ & 0.36788 \\ 
    $t=2$ & 0.14665 & 0.18394 \\
    $t=3$ & 0.07703 & 0.08421 & 0.12263 \\
    $t=4$ & 0.04444 & 0.04662 & 0.05664 & 0.09197 \\
    $t=5$ & 0.02657 & 0.02732 & 0.03092 & 0.04152 & 0.07358 \\
    $t=6$ & 0.01608 & 0.01635 & 0.01773 & 0.02184 & 0.03214 & 0.06131 \\
    $t=7$ & 0.00974 & 0.00984 & 0.01038 & 0.01204 & 0.01614 & 0.02583 & 0.05255\\
    \bottomrule
\end{tabular}
\caption{Approximations of the radii of convergence of the generating functions counting $k$-connected chordal graphs with tree-width at most $t$ for small values of $t$ and $k$.}\label{tab:values_rho}
\end{table}

Let us mention a recent result giving an estimate $cn^{-5/2} \gamma^n n!$ for the number of labelled \emph{planar} chordal graphs with $\gamma \approx 11.89$ \cite{CNR23}. 
Is is easy to see that the class of chordal graphs with tree-width at most three is exactly the same as the class of chordal graphs not containing $K_5$ as a minor, whose asymptotic growth is, according to Theorem \ref{thm:enumeration} and the table above, of the form $cn^{-5/2} \delta^n n!$ with $\delta = 1/\rho_{3,1}\approx 12.98$. 

Furthermore, notice that if we denote by $C(x)$ and $B(x)$, respectively, the generating functions of connected and 2-connected graphs in $\G_{t,0}$, then Equation \ref{eq:recursive} reads for $k=1$
\begin{align*}
    C'(x) = \exp(B'(xC'(x)).
\end{align*}
If $\rho_C$ and $\rho_B$ are the singularities of $C(x)$ and $B(x)$, respectively, the condition for being subcritical is that $\rho_C C'(\rho_C) < \rho_B$, so that the singularity of $C(x)$ arises as a branch-point in the former equation and is not inherited by that of $B(x)$; in our case this condition is safistied because of Lemma~\ref{lem:implicit}. 

Since the number  of all chordal graphs grows like $2^{n^2/4}$, we know that the singularity $\rho_t=\rho_{t,1}$ of  chordal graphs with tree-wdith at most $t$ goes to 0 as $t\to \infty$.
The question is at which rate $\rho_t \to 0$ as $t\to\infty$.
Since the exponential growth of $t$-trees is $(etn)^n$, we have $\rho_t = O(1/t)$. 
And since the growth of all graphs of tree-width at most $t$ is at most $(2^ttn)^n$, we also have $\rho_t = \Omega(1/(t2^t))$.
We leave as an open problem to narrow the gap between the upper and lower bound. 
Heuristic arguments suggest that $\rho_t$ decreases exponentially in $t$. 

As a final question, we consider letting $t=t(n)$ grow  with $n$. 
Recall that a  class of labelled graphs is small when the number of graphs in the class grows at most like $c^n n!$ for some $c>0$, and large  otherwise.
We know that  the class of all chordal graphs is large, while the class of chordal graphs with tree-width at most $t$ is small for fixed $t$. 
Let us see that if  $t = (1+\epsilon)(\log n)$  then the class is large for each $\epsilon>0$. 
A graph is split if the vertex set can be partitioned into a clique and an independent set. 
It is well-known and easy to prove that split graphs are chordal. 
Consider split graphs with a clique of size $t=(1+\epsilon)\log n$ and the complement an independent set, so that he largest clique is of size at most $t+1$ and the tree-width at most $t$. 
Every edge between the clique and the complement can be chosen independently, hence there are at least 
$$2^{(1+\epsilon)\log n (n-(1+\epsilon)\log n)}$$ 
such graphs, a quantity that  grows faster than  $c^n n!$ for every $c>0$. 
We leave as an open problem to determine at which order of magnitude between $t=O(1)$ and $t=\log n$ the class ceases to be small.

\subsection*{Acknowledgements}

We would like to thank the anonymous referee for carefully reading a first version of this work and for their suggestions and comments that improved the presentation of the paper.

Additionally, we gratefully acknowledge earlier discussions with Juanjo Ru\'e and Dimitrios Thilikos on the problem of counting chordal graphs with bounded tree-width.

The authors acknowledge support from the Marie Curie RISE research network ``RandNet'' MSCA-RISE-2020-101007705.
Moreover, M.D. was supported by the Special Research Program SFB F50-02 ``Algorithmic and Enumerative Combinatorics'', and by the project P35016 ``Infinite Singular Systems and Random Discrete Objects'' of the FWF (Austrian Science Fund).
Additionally, M.N. and C.R. acknowledge the financial support of the Spanish State Research Agency through projects MTM2017-82166-P and PID2020-113082GB-I00, while M.N. acknowledges support from the Severo Ochoa and Mar\'ia de Maeztu Program for Centers and Units of Excellence (CEX2020-001084-M), and C.R. acknowledges support from the grant Beatriu de Pin\'os BP2019, funded by the H2020 COFUND project No 801370 and AGAUR (the Catalan agency for management of university and research grants).

%
%
\bibliographystyle{abbrv}
\bibliography{biblio_chordal_tw}

\end{document}